\documentclass[13pt, reqno]{amsart}
\usepackage{amsmath}
\usepackage{amscd}
\usepackage{amsthm}
\usepackage{amssymb} \usepackage{latexsym}
\usepackage{eufrak}
\usepackage{euscript}
\usepackage{epsfig}
\usepackage{graphics}
\usepackage{array}
\usepackage{enumerate}
\usepackage{color}
\usepackage{wasysym}
\usepackage{pdfsync}
\usepackage{stmaryrd}

\newcommand{\bel}[1]{\begin{equation}\label{#1}}

\newcommand{\be}{\begin{equation}}

\newcommand{\CCD}{\mathrm{CD}}

\newcommand{\ba}{\begin{eqnarray}}
\newcommand{\ea}{\end{eqnarray}}

\newcommand{\qe}{\end{equation}}
\newcommand{\R}{{\mathbb R}}
\newcommand{\N}{{\mathbb N}}
\newcommand{\Z}{{\mathbb Z}}

\newcommand{\y}{(\Omega)}

\newcommand{\BG}{\mathcal{BG}}

\newcommand{\eg}{\begin{example}}
\newcommand{\egd}{\end{example}}
\newcommand{\tm}{\begin{thm}}
\newcommand{\tmd}{\end{thm}}
\newcommand{\co}{\begin{coro}}
\newcommand{\cod}{\end{coro}}
\newcommand{\enu}{\begin{enumerate}}
\newcommand{\enud}{\end{enumerate}}
\newcommand{\rmk}{\begin{rem}}
\newcommand{\rmkd}{\end{rem}}

\theoremstyle{theorem}
\newtheorem{thm}{Theorem}[section]
\newtheorem{prop}[thm]{Proposition}
\theoremstyle{example}
\newtheorem{example}[thm]{Example}
\newtheorem{coro}[thm]{Corollary}
\theoremstyle{lemma}
\newtheorem{lemma}[thm]{Lemma}
\theoremstyle{definition}
\newtheorem{defi}[thm]{Definition}
\theoremstyle{proof}

\theoremstyle{remark}
\newtheorem{rem}[thm]{Remark}
\theoremstyle{remark}

\begin{document}

\title[Graphs with nonnegative curvature outside a finite subset]{Graphs with nonnegative curvature outside a finite subset, harmonic functions and number of ends}

\author{Bobo Hua}
\address{Bobo Hua: School of Mathematical Sciences, LMNS,
Fudan University, Shanghai 200433, China; Shanghai Center for
Mathematical Sciences, Fudan University, Shanghai 200438,
China.}
\email{bobohua@fudan.edu.cn}

\author{Florentin M\"unch}
\address{Florentin M\"unch: Max Planck Institute for Mathematics in the Sciences, Leipzig 04103, Germany}
\email{florentin.muench@mis.mpg.de}
\begin{abstract}
We study graphs with nonnegative Bakry-\'Emery curvature or Ollivier curvature outside a finite subset. For such a graph, via introducing the discrete Gromov-Hausdorff convergence we prove that the space of bounded harmonic functions is finite dimensional, and as a corollary the number of non-parabolic ends is finite. 
\end{abstract}

\maketitle




\section{Introduction}

For a complete, noncompact Riemannian manifold with nonnegative Ricci curvature, Yau \cite{Yau75} proved that the Liouville theorem for positive harmonic functions, i.e. any  positive harmonic function is constant. Later, Cheng and Yau \cite{ChengYau75} proved a quantitative gradient estimate for positive harmonic functions, see a parabolic version by Li and Yau \cite{LiYau86}.
 Liouville theorems for harmonic functions have received much attention in the literature, e.g. \cite{KaimanovichVershik83,Sullivan83,Anderson83,LiTam87,Lyons87,Grigoryan90,Benjamini91,Grigoryan91,Saloff92,Kaimanovich96,WangFY02,Erschler04,Woess09,Li12,Brighton13}.
 
Liouville type theorems for harmonic functions have been generalized to manifolds with nonnegative Ricci curvature outside a compact set. Let $M$ be such a manifold. Donnelly \cite{Donnelly86}  first proved that the space of bounded harmonic functions on $M$ is finite dimensional. Cheng \cite{Chengpreprint} and Li and Tam \cite[Theorem~3.2]{LiTam92} gave a quantitative estimate of the dimension
of bounded harmonic functions via the dimension of $M$, the
diameter of the set where $M$ has negative Ricci curvature, and the lower
bound of the Ricci curvature on $M,$ see e.g. \cite[Theorem~7.4]{Li12}.

In this paper, we study bounded harmonic functions on graphs with nonnegative curvature outside a finite subset of vertices (or edges). We recall the setting of weighted graphs. Let $(V,E)$ be a locally finite, simple, undirected graph. Two vertices $x,y$ are called neighbours, denoted by $x\sim y$, if there is an edge connecting $x$ and $y,$ i.e. $\{x,y\}\in E.$ For any $x\in V,$ we denote by
$\deg(x)=\sharp\{y\in V: y\sim x\}$ the combinatorial degree of the vertex $x.$ A graph is called connected if for any $x,y\in V$ there is a path $\{x_i\}_{i=0}^n\subset V$ connecting $x$ and $y,$ i.e. 
$$x = z_0 \sim. . . \sim z_n = y.$$ In this paper, we always consider connected graphs.
We denote by $$d(x, y) :=
\inf\{n | x = z_0 \sim. . . \sim z_n = y\}$$ the combinatorial graph distance between vertices $x$ and $y.$ For any $R\in \N,$ we write $$B_R(x):=\{y\in V: d(y,x)\leq R\}\ (\mathrm{resp}.\ S_R(x):=\{y\in V: d(y,x)=R\})$$ for the ball (resp. the sphere) of radius $R$ centered at $x.$ 
Let $$w: E\to (0,\infty),\quad \{x,y\}\mapsto w(x,y)=w(y,x),$$ be an edge weight function, and $$m: V\to (0,\infty),\quad x\mapsto m(x),$$ be a vertex weight function. We call the quadruple $G=(V,E,m,w)$ a \emph{weighted graph}.

For a weighted graph $G$ and any function $f:V\to \R,$ the Laplace operator $\Delta$ is defined as
$$\Delta f (x):= \sum_{y\in V:y\sim x}\frac{w(x,y)}{m(x)}\left(f(y)-f(x)\right), \quad\forall x\in V.$$
{For $\Omega\subset V,$ a function $f$ on $V$ is called harmonic (resp. superharmonic, subharmonic) on $\Omega$ if $\Delta f=0$ (resp. $\Delta f\leq 0, \Delta f\geq 0$) on $\Omega.$  We denote by $\mathcal{H}_0(G)$ the space of bounded harmonic functions on $V$ of $G.$}

We introduce the curvature notions on graphs. Let $K\in \R$ and $n\in(0,\infty].$ As is well-known, for a Riemannian manifold $M,$ the Ricci curvature is bounded below by $K$ and the dimension is bounded above by $n,$ if and only if
\begin{equation}\label{eq:eqb}\frac12\Delta_M |\nabla f|^2\geq \frac{1}{n}(\Delta f)^2+\langle \nabla f,\nabla \Delta_M f\rangle+K|\nabla f|^2,\quad\forall \ f\in C_c^\infty(M),\end{equation} where $\Delta_M$ is the Laplace-Beltrami operator on $M$ and $\nabla\cdot$ is the gradient of a function.
For a general Markov semigroup, Bakry and \'Emery \cite{BakryEmery85,Bakry87,BakryGentilLedoux} introduced the $\Gamma$-calculus, and defined the curvature dimension condition mimicking \eqref{eq:eqb}, denoted by $\CCD(K,n).$
For weighted graphs, this condition is called the Bakry-\'Emery curvature condition, introduced by \cite{Elw91,Schmuck96,LinYau10} independently, see Subsection~\ref{subsec:curvature}.
It is proved in \cite{hua2019Liouville} that the Liouville theorem for bounded harmonic functions holds for a weighted graph satisfying the $\CCD(0,\infty)$ condition.

Ollivier \cite{Ollivier09} introduced another curvature notion on graphs via the optimal transport, which was later modified by \cite{LinLuYau11,munch2019Ollivier}. In this paper, we call the modified curvature notion the Ollivier curvature on graphs, see Subsection~\ref{subsec:curvature}. {The Ollivier curvature is closely related to the Forman curvature for cell complexes, see \cite{JostMun21,tee2021enhanced}, and \cite{forman2003bochner} for Forman's original work.} The Liouville theorem for bounded harmonic functions on a weighted graph satisfying nonnegative Ollivier curvature was proved by \cite{jost2019liouville,munch2019non}. 
It was used by Salez in \cite{salez2022sparse} to show that graphs with non-negative curvature cannot be expanders, and this result even allows negative curvature on a ball of fixed radius.

Since the above curvature notions are local conditions, it is natural to consider a weighted graph with nonnegative curvature outside a finite subset in the sense of either Bakry-\'Emery or Ollivier.
However, the arguments for proving Liouville theorems in \cite{hua2019Liouville,jost2019liouville,munch2019non} do not work in this setting, since the global information of the graph, i.e.  nonnegative curvature everywhere, is crucially used. {It arises a common difficulty for semigroup approaches to analyze a space involving some negative curvature, see \cite{Paeng12,munch2018perpetual,liu2019distance,munch2020spectrally,Mun22F}.} 
Even worse, the bounded Liouville property clearly fails when allowing some negative curvature as one can see from gluing together two copies of $\Z^3$ at a single vertex.
To circumvent the difficulty, we introduce a discrete analog of the Gromov-Hausdorff convergence. In metric geometry, Gromov \cite{Gromov81} initiated the Gromov-Hausdorff convergence for metric spaces, which extends compactness arguments to a very general setting, see e.g. \cite{Burago01}. In this paper, we modify the Gromov-Hausdorff convergence for weighted graphs equipped with combinatorial distances, and prove the Liouville type results for bounded harmonic functions on graphs with nonnegative curvature outside a finite subset. See  \cite{BenSchramm01,AldousLyons07} for other related convergence notions on graphs.

\begin{defi}\label{def:bg1} We say that a weighted graph $G=(V,E,m,w)$ has \emph{bounded geometry} if there is a positive constant $C$ such that 
\begin{equation}\deg(x)\leq C,\ \forall x\in V,
\end{equation}
\begin{equation}C^{-1}\leq m(x)\leq C,\ \forall x\in V,
\end{equation} and 
\begin{equation}C^{-1}\leq w(x,y)\leq C,\ \forall x,y\in V, x\sim y.
\end{equation} We denote by $\mathcal{BG}(C),$ $\mathcal{BG}$ in short if the constant $C$ is clear in the context, the class of graphs satisfying the above conditions.
\end{defi}

The following are the main results of the paper.
\begin{thm}\label{thm:mainBE} Let $G=(V,E,m,w)$ be a weighted graph with bounded geometry. If $G$ satisfies $\CCD(0,\infty)$ outside $B_{R_0}(x_0)$ for some $x_0\in V$ and $R_0\in \N,$ see Definition~\ref{def:outside}, then 
$$\dim \mathcal{H}_0(G)\leq \sharp S_{R_0+1}(x_0)<\infty,$$ where $\sharp(\cdot)$ denotes the cardinality of a set.
\end{thm}

\begin{thm}\label{thm:mainOllivier} Let $G=(V,E,m,w)$ be a weighted graph with bounded geometry. If $G$ has nonnegative Ollivier curvature outside $B_{R_0}(x_0)$ for some $x_0\in V$ and $R_0\in \N,$  see Definition~\ref{def:outside}, then 
$$\dim \mathcal{H}_0(G)\leq \sharp S_{R_0+1}(x_0)<\infty.$$
\end{thm}
The proof strategies are as follows. For the first step, we prove that for any bounded harmonic function $u$ on $G$ the gradient of the function, $\Gamma(u)$ (or $|\nabla u|$), tends to zero at infinity. To prove that, we use the pointed Gromov-Hausdorff convergence on graphs, see Section~\ref{Sec:Preliminaries}, and the contradiction argument. Suppose that it is not true, then there is a sequence of vertices $\{x_i\}_{i=1}^\infty$ tends to infinity such that $\liminf_{i\to\infty}\Gamma(u)(x_i)>0.$ Considering the sequence of rooted graphs $\{(G,x_i)\}_{i=1}^\infty,$ by the compactness, up to subsequence there exist a pointed Gromov-Hausdorff limit $(G_\infty,x_\infty)$ and a limiting harmonic function $u_\infty$ on $G_\infty.$ By the curvature assumption of $G,$ one can show that $G_\infty$ has nonnegative curvature everywhere. By the Liouville property (see \cite{hua2019Liouville, jost2019liouville,munch2019non}), $u_\infty$ is constant which contradicts that  $\liminf_{i\to\infty}\Gamma(u)(x_i)>0.$ For the second step, we show that the maximum of the gradient of $u$ is attained at the boundary of the negatively curved subset. This can be derived from the subharmonicity of $\Gamma(u)$ outside a finite subset for the case of Bakry-\'Emery curvature and the maximum principle. A modified argument is needed for the case of Ollivier curvature, see Lemma~\ref{lem:MaxPrincipleOllivierGradient}. As a consequence, we obtain some kind of unique continuation property for bounded harmonic functions. In the last step, we adopt linear algebra to count the dimension of bounded harmonic functions. This argument for proving the finite dimensionality of bounded harmonic functions seems new even for the Riemannian case.

As is well-known, the space of bounded harmonic functions is related to the number of non-parabolic ends in Riemannian geometry, see  \cite{LiTam87,LiTam92,SungTamWang00,Li12}. We extend this to the setting of graphs and obtain the following corollary.
\begin{coro}\label{coro:finiteend} Let $G$ be a graph as in Theorem~\ref{thm:mainBE} or Theorem~\ref{thm:mainOllivier}. Then the number of non-parabolic ends of $G$ is at most $\sharp S_{R_0+1}(x_0).$
\end{coro}
\begin{rem} {It was proved in \cite{HuaMun21} that the number of infinite-volume ends of a graph with nonnegative Ollivier is at most two.}
\end{rem}

The paper is organized as follows: In the next section, we introduce the discrete Gromov-Hausdorff convergence, and curvature notions on graphs. In Section~\ref{sec:harmonic functions}, we prove the main theorems, Theorem~\ref{thm:mainBE} and Theorem~\ref{thm:mainOllivier}. The last section is devoted to the theory of harmonic functions  and  ends of graphs.

\section{preliminaries}\label{Sec:Preliminaries}

\subsection{Discrete Gromov-Hausdorff convergence}
Let $G$ be a combinatorial graph $(V,E)$ or a weighted graph $(V,E,m,w).$ For any vertex $x\in V,$
we call the pair $(G,x)$ a rooted graph on $G$ with the root $x.$

Let $C$ be a fixed positive constant. We denote by $\BG^p(C)$ the set of rooted graphs $(G,x)$ where $G\in \BG(C)$ and $x$ is a vertex in $G,$ and write $\BG^p$ in short if the constant $C$ is evident in the context. We define the topology on $\BG^p,$ which is a discrete analog of pointed Gromov-Hausdorff convergence of metric spaces.  Let $(V_i,E_i)$ be two combinatorial graphs and $p_i\in V_i,$ $i=1,2.$ If $\varphi:(V_1,E_1)\to (V_2,E_2)$ is a graph isomorphism and $\varphi(p_1)=p_2,$ then we call it a rooted graph isomorphism between $(V_1,E_1,p_1)$ and $(V_2,E_2,p_2),$ denoted by
$$\varphi:(V_1,E_1,p_1)\to (V_2,E_2,p_2).$$
\begin{defi} Let $(G_i,p_i), 1\leq i\leq \infty$ and $(G_\infty,p_\infty)$ be finite rooted graphs. We say that $(G_i,p_i)$ pointed Gromov-Hausdorff converges to $(G_\infty,p_\infty),$ denoted by $(G_i,p_i)\xrightarrow[]{pGH}(G_\infty,p_\infty),$
if there exists $N$ such that for any $i\geq N,$ there exists rooted graph isomorphism $$\varphi_{i}:(V_i,E_i,p_i)\to (V_\infty,E_\infty,p_\infty)$$ with the properties
\begin{equation}\label{def:mcv}m_i(\varphi_{i}^{-1}(x))\to m_\infty(x), \forall x\in V_\infty,\end{equation} and \begin{equation}\label{def:wcv}w_i(\varphi_{i}^{-1}(x),\varphi_{i}^{-1}(y))\to w_\infty(x,y), \forall x,y\in V_\infty, x\sim y.\end{equation}
\end{defi}
\begin{rem}\begin{enumerate}[(i)]\item The convergence $(G_i,p_i)\xrightarrow[]{pGH}(G_\infty,p_\infty)$ can be regarded as the discrete analog of pointed Gromov-Hausdorff convergence, where $G_i$ and $G_\infty$ are endowed with combinatorial distances. In our definition, we further require the convergence of weights on edges and vertices.
\item By the definition, if $(G_i,p_i)\xrightarrow[]{pGH}(G_\infty,p_\infty),$ then for any $x\in V_\infty,$
$$(G_i,\varphi_{i}^{-1}(x))\xrightarrow[]{pGH}(G_\infty,x).$$
\end{enumerate}
\end{rem}

\begin{defi} Let $\{(G_i,p_i)\}_{i=1}^\infty\subset \BG^p$ and $(G_\infty,p_\infty)\in \BG^p.$ We say that $(G_i,p_i)$ pointed Gromov-Hausdorff converges to $(G_\infty,p_\infty),$ denoted by $$(G_i,p_i)\xrightarrow[]{pGH}(G_\infty,p_\infty),$$ if
for any $R\in \N,$ $$(B_R^{G_i}(p_i),p_i)\xrightarrow[]{pGH}(B_R^{G_\infty}(p_\infty),p_\infty), \quad i\to\infty,$$ where $B_R^{G_i}(p_i)$ (resp. $B_R^{G_\infty}(p_\infty)$) is understood as the induced graph with inherited weights on $B_R^{G_i}(p_i)$ in $G_i$ (resp. on $B_R^{G_\infty}(p_\infty)$ in $G_\infty$). For sufficiently large $i,$ we denote the rooted graph isomorphisms by
$$\varphi_{i,R}:(B_R^{G_i}(p_i),p_i)\to(B_R^{G_\infty}(p_\infty),p_\infty).$$
\end{defi}

\begin{defi}\label{def:pGHGlobalConv} Let $\{(G_i,p_i)\}_{i=1}^\infty\cup \{(G_\infty,p_\infty)\}\subset \BG^p$ satisfy that $$(G_i,p_i)\xrightarrow[]{pGH}(G_\infty,p_\infty), i\to \infty.$$ Let $u_i:V_i\to\R$ and $u_\infty: V_\infty\to\R.$ We say that $u_i$ converges to $u_\infty$ if for any $R\in \N,$ and sufficiently large $i,$
$$u_i\circ \varphi_{i,R}^{-1}   {\longrightarrow} u_\infty,\quad \mathrm{on}\ B_R^{G_\infty}(p_\infty),\quad i\to\infty.$$
\end{defi}

We prove the discrete analog of Gromov's compactness theorem for pointed Gromov-Hausdorff convergence on graphs.

\begin{thm}\label{thm:comp1} The space $\BG^p$ equipped with the pointed Gromov-Hausdorff topology is sequentially compact, i.e. for any sequence $\{(G_i,p_i)\}_{i=1}^\infty\subset \BG^p,$ there exist a subsequence, still denoted by $\{(G_i,p_i)\}_{i=1}^\infty$, and $(G_\infty,p_\infty)\in \BG^p$ such that
$$(G_i,p_i)\xrightarrow[]{pGH}(G_\infty,p_\infty).$$ 
\end{thm}
\begin{proof} Let $R\in \N$ and $C_1>0.$ We set $$\mathcal{G}_{R,C_1}^p:=\{(V,E,x_0): x_0\in V, \deg(x)\leq C_1, d(x,x_0)\leq R \mbox{ for all } x \in V\}/\sim,$$ where $\sim$ denotes the equivalence condition of rooted graph isomorphism. Note that there are only finitely many combinatorial structure in the class $\mathcal{G}_{R,C_1}^p.$ 

Let $\BG^p=\BG^p(C)$ for $C>0.$ Choose $C_1\geq C.$ For the sequence $\{(G_i,p_i)\}_{i=1}^\infty\subset \BG^p,$ $$\{(\widehat{B}_R^{G_i}(p_i),p_i)\}_{i=1}^\infty\subset \mathcal{G}_{R,C_1}^p$$ where $\widehat{B}_R^{G_i}(p_i)$ denote the induced combinatorial graphs on $B_R^{G_i}(p_i).$
Hence there is a subsequence, still denoted by $\{(G_i,p_i)\}_{i=1}^\infty,$ and a rooted graph $(G_R,p_\infty)\in\mathcal{G}_{R,C_1}^p$ such that
$(\widehat{B}_R^{G_i}(p_i),p_i)$ is rooted-graph isomorphic to $(G_R^\infty,p_\infty)$ with $G_R^\infty=(V_R^\infty,E_R^\infty)$ for sufficiently large $i,$ i.e. \begin{equation}\label{eq:comp1}\varphi_{i,R}:(\widehat{B}_R^{G_i}(p_i),p_i)\to (G_R^\infty,p_\infty).\end{equation}
Note that there are finitely many edges and vertices in $G_R^\infty,$ and by $\{(G_i,p_i)\}_{i=1}^\infty\subset \BG^p$ and Definition \ref{def:bg1},
$$C^{-1}\leq m_i(\varphi_{i,R}^{-1}(x))\leq C,\quad \mathrm{and}$$ $$C^{-1}\leq w_i(\varphi_{i,R}^{-1}(x),\varphi_{i,R}^{-1}(y))\leq C,\quad \forall x,y\in V_R, x\sim y.$$ Hence there is a subsequence, still denoted by $\{(G_i,p_i)\}_{i=1}^\infty,$ such that
$$m_i(\varphi_{i,R}^{-1}(x))\to a_R(x)\in [C^{-1},C], \quad w_i(\varphi_{i,R}^{-1}(x),\varphi_{i,R}^{-1}(y))\to b_R(x,y)\in [C^{-1},C]\quad \forall x,y\in V_R, x\sim y.$$  We set the weights on $G_R^\infty$ by
$$m(x)=a_R(x), \forall x\in V_R^\infty\ \mathrm{and} \quad w(x,y)=b_R(x,y),  \forall x,y\in V_R^\infty, x\sim y.$$
This yields that for the subsequence $\{(G_i,p_i)\}_{i=1}^\infty,$ 
$$(B_R^{G_i}(p_i),p_i)\xrightarrow[]{pGH}(G_R^\infty,p_\infty), \quad i\to\infty,$$ where $B_R^{G_i}(p_i)$ are the induced subgraphs with inherited weights on $B_R^{G_i}(p_i)$ from $G_i.$

Now we consider $R+1.$ We fix the rooted graph isomorphism in \eqref{eq:comp1}.
Consider the induced subgraphs $\{(\widehat{B}_{R+1}^{G_i}(p_i),p_i)\}_{i=1}^\infty.$ There are finitely many combinatorial structures in $\mathcal{G}_{R+1,C_1}^p$ when fixing the structure for $(\widehat{B}_{R}^{G_i}(p_i),p_i).$ So that there is a subsequence, denoted by  $\{(G_{i_j},p_{i_j})\}_{j=1}^\infty,$ and $(G_{R+1}^\infty,p_\infty)\in\mathcal{G}_{R+1,C_1}^p$ such that for sufficiently large $j$ there exist rooted graph isomorphisms 
\begin{equation*}\varphi_{j,R+1}:(\widehat{B}_{R+1}^{G_{i_j}}(p_{i_j}),p_{i_j})\to (G_{R+1}^\infty,p_\infty)\end{equation*}
satisfying $\varphi_{j,R+1}\big|_{B_{R}^{G_{i_j}}(p_{i_j})}=\varphi_{i_j,R}\big|_{B_{R}^{G_{i_j}}(p_{i_j})}.$ That is, in our construction we can require that $\varphi_{j,R+1}$ extends $\varphi_{i_j,R}.$ So that there is a rooted graph isomorphism $(G_R^\infty,p_\infty)\to (\widehat{B}_R^{G_{R+1}^\infty}(p_\infty),p_\infty).$ Hence we can identify $G_R^\infty$ as a subgraph $\widehat{B}_R^{G_{R+1}^\infty}(p_\infty)$ of $G_{R+1}^\infty.$ By the same argument for $R,$ there are a subsequence, still denoted by 
 $\{(G_i,p_i)\}_{i=1}^\infty$ for simplicity, and weights on $G_{R+1}^\infty$ such that
$$(B_{R+1}^{G_i}(p_i),p_i)\xrightarrow[]{pGH}(G_{R+1}^\infty,p_\infty), \quad i\to\infty.$$ Note that one can choose a subsequence s.t.
\begin{equation}\label{eq:gc1}m_R^\infty(x)=m_{R+1}^\infty(x), \forall x\in V_{R}^\infty,\quad w_{R}^\infty(x,y)=w_{R+1}^\infty(x,y),\quad \forall x,y\in V_{R}^\infty, x\sim y.\end{equation}

By the induction for $R=1,2,\cdots$ and the diagonal argument, there exists a subsequence, still denoted by $\{(G_i,p_i)\}_{i=1}^\infty,$ such that for all $R \in \N$,
$$(B_R^{G_i}(p_i),p_i)\xrightarrow[]{pGH}(G_R^\infty,p_\infty), \quad i\to\infty.$$
Note that $(G_R^\infty,p_\infty)$ is rooted-graph isomorphic to $\widehat{B}_R^{G_{R+1}^\infty}(p_\infty)$ of $G_{R+1}^\infty$ with proper edge and vertex weights satisfying \eqref{eq:gc1}. Since the graph structures are consistent, we may find a rooted graph $(G_\infty, p_\infty)$ such that
$(G_R^\infty,p_\infty)$ is rooted-graph isomorphic to $B_R^{G_{\infty}}(p_\infty)$ in $(G_{\infty},p_\infty)$ for all $R\in \N.$ We can endow $(G_\infty, p_\infty)$ with the edge and vertex weights by the property \eqref{eq:gc1}.   By the definition of pointed Gromov-Hausdorff convergence,
$$(G_i,p_i)\xrightarrow[]{pGH}(G_\infty,p_\infty), \quad i\to\infty.$$ This proves the theorem.
\end{proof}


For any function $u:V\to\R$ and $\Omega\subset V,$ we denote by $$\|u\|_{\ell^\infty(\Omega)}:=\sup_{x\in \Omega}|u(x)|$$ the $\ell^\infty$ norm of $u$ on $\Omega.$ If $\Omega=V,$ then we write $\|u\|_\infty$ for simplicity.

\begin{thm}\label{thm:comp2} Let $\{(G_i,p_i)\}_{i=1}^\infty\subset \BG^p$ and $u_i: V_i\to\R.$ Suppose that
$\sup_i\|u_i\|_\infty<\infty,$ then there exist a subsequence, still denoted by $\{(G_i,p_i)\}_{i=1}^\infty,$ $(G_\infty,p_\infty)\in \BG^p,$ and $u_\infty:V_\infty\to\R,$ such that
$$(G_i,p_i)\xrightarrow[]{pGH}(G_\infty,p_\infty), \quad\mathrm{and}\quad u_i\to u_\infty,\quad i\to \infty$$
where the convergence of $u_i$ is defined in Definition~\ref{def:pGHGlobalConv}. 
\end{thm}
\begin{proof} The result follows verbatim from the proof of Theorem~\ref{thm:comp1} by regarding $u_i$ as $m_i$ therein.
\end{proof}

\subsection{Curvature conditions}\label{subsec:curvature}

We introduce the $\Gamma$-calculus and Bakry-\'Emery's curvature dimension conditions on graphs following \cite{Elw91,Schmuck96,LinYau10}.


First we define two natural bilinear forms associated to the Laplacian $\Delta$. We denote by $C(V)$ the set of functions on $V.$ 
\begin{defi}\label{d:carre du}
  The gradient form $\Gamma,$ called the ``carr\'e du champ" operator, is defined by, for $f,g\in C(V)$ and $x\in V$,
  \begin{eqnarray*}\Gamma(f,g)(x)&=&\frac12(\Delta(fg)-f\Delta g-g\Delta f)(x)\\&=&\frac{1}{2m(x)}
  \sum_{y\sim x}w(x,y)(f(y)-f(x))(g(y)-g(x)).\end{eqnarray*} For simplicity, we write $\Gamma(f):=\Gamma(f,f).$ The $\Gamma_2$ operator is defined as $$\Gamma_2(f,g)=\frac12(\Delta\Gamma(f,g)-\Gamma(f,\Delta g)-\Gamma(g,\Delta f)),$$ and we write $\Gamma_2(f):=\Gamma_2(f,f)=\frac{1}{2}\Delta \Gamma(f)-\Gamma(f,\Delta f).$
\end{defi}


Now we introduce curvature dimension conditions on graphs.
\begin{defi}\label{d:curvature dimension}
 Let $K\in \R, n\in (0,\infty].$ We say a weighted graph $G$ satisfies the $\CCD(K,n)$ condition at $x\in V,$ denoted by $\CCD(K,n,x),$ if for any $f\in C(V)$,
  \begin{equation}\label{CDcondition}\Gamma_2(f)(x)\geq \frac1n(\Delta f(x))^2+K\Gamma (f)(x).\end{equation} For given $n$ and $x,$ we denote by $\mathcal{K}^{G}_{n}(x)$ the maximal constant $K$ such that the above inequality holds for all $f.$ We call it the Bakry-\'Emery curvature at $x.$ 
  We say that a graph satisfies the $\CCD(K,n)$ condition if $\CCD(K,n,x)$ holds for all $x\in V.$ 
\end{defi}

In contrast to Bakry-\'Emery curvature based on the $\Gamma$-calculus, the Ollivier curvature is based on the Wasserstein distance via the optimal transport \cite{Ollivier09}, see also \cite{LinLuYau11,munch2019Ollivier}. By \cite[Theorem~2.1]{munch2019Ollivier}, the Ollivier curvature $\kappa(x,y),$ $x,y\in V,$ can be calculated as
\begin{equation}\label{eq:ofc}
\kappa(x,y)=\inf_{\nabla_{yx}f=1,|\nabla f|_\infty=1}\nabla_{xy}\Delta f.
\end{equation}
where
$\nabla_{xy} f  :=\frac{f(x)-f(y)}{d(x,y)}$ 
and
$|\nabla f|_{\infty} 
:= \sup_{x\sim y} \nabla_{xy}f$.

For any subset $\Omega$ in $V,$ we denote by $$E_\Omega:=\{e=\{x,y\}\in E: x\in \Omega, y\in \Omega\}$$ the set of edges whose end-vertices are contained in $\Omega.$
We define the weighted graph with nonnegative curvature outside a finite subset as follows.
\begin{defi}\label{def:outside}
Let $G$ be a weighted graph and $\Omega$ be a finite subset of $V.$ We say that $G$ satisfies $\CCD(0,\infty)$ outside $\Omega$ if $$\mathcal{K}^{G}_{\infty}(x)\geq 0,\quad \forall x\in V\setminus \Omega.$$ We say that $G$ has nonnegative Ollivier curvature outside $\Omega$ if  $$\kappa(x,y)\geq 0,\quad \forall e=\{x,y\}\in E_{V\setminus \Omega}.$$
\end{defi}

Since these curvature conditions are local properties, we can show that the curvature lower bound is preserved under the pointed Gromov-Hausdorff convergence.

\begin{prop}\label{prop:lowerbound1}Let $\{(G_i,p_i)\}_{i=1}^\infty\cup\{(G_\infty,p_\infty)\}\subset \BG^p$ satisfy $$(B_2^{G_i}(p_i),p_i)\xrightarrow[]{pGH}(B_2^{G_\infty}(p_\infty),p_\infty), \quad i\to\infty,$$ with the rooted graph isomorphism $\varphi_i$ for sufficiently large $i.$ Let $K\in \R.$ Suppose that for sufficiently large $i,$
$\mathcal{K}_{\infty}^{G_i}(p_i)\geq K,$ then $$\mathcal{K}_{\infty}^{G_\infty}(p_\infty)\geq K.$$ 
\end{prop}
\begin{proof} Note that the curvature $\mathcal{K}_{\infty}^{G_\infty}(p_\infty)$ is determined by the structure of $B_2^{G_\infty}(p_\infty).$ For any function $f:B_2^{G_\infty}(p_\infty)\to\R,$ we set $f_i=f\circ \varphi_i: B_2^{G_i}(p_i)\to\R.$ By the curvature condition $\mathcal{K}_{\infty}^{G_i}(p_i)\geq K,$ $$\Gamma_2^{G_i}(f_i)-K\Gamma^{G_i}(f_i)\geq 0.$$ By passing to the limit, $i\to\infty,$ noting that \eqref{def:mcv} and \eqref{def:wcv},
 $$\Gamma_2^{G_\infty}(f)-K\Gamma^{G_\infty}(f)\geq 0.$$ This yields the result.
\end{proof}

\begin{prop}\label{prop:lowerbound2}Let $\{(G_i,p_i)\}_{i=1}^\infty\cup\{(G_\infty,p_\infty)\}\subset \BG^p$ satisfy $$(B_3^{G_i}(p_i),p_i)\xrightarrow[]{pGH}(B_3^{G_\infty}(p_\infty),p_\infty), \quad i\to\infty,$$ with the rooted graph isomorphism $\varphi_i$ for sufficiently large $i.$ Let $K\in \R$ and $x\in V_\infty,$ $x\sim p_\infty.$ Suppose that for sufficiently large $i,$
$\kappa^{G_i}(p_i,\varphi_i^{-1}(x))\geq K,$ then $$\kappa^{G_\infty}(p_\infty,x)\geq K.$$ 
\end{prop}
\begin{proof} Note that the curvature condition of $\mathcal{K}_{\infty}^{G_\infty}(p_\infty,x)$ is determined by the structure of $B_3^{G_\infty}(p_\infty).$ By the property \eqref{eq:ofc}, for any function $f:B_3^{G_\infty}(p_\infty)\to\R$ satisfying $\nabla_{xp_\infty}f=1$ and $|\nabla f|_\infty=1,$ we set $f_i=f\circ \varphi_i: B_3^{G_i}(p_i)\to\R.$ Note that $\nabla_{\varphi_i^{-1}(x)p_i}f_i=1$ and $|\nabla f_i|_\infty=1.$ By the curvature condition $\kappa^{G_i}(p_i,\varphi_i^{-1}(x))\geq K,$ $$\nabla_{p_i\varphi_i^{-1}(x)}\Delta^{G_i} f_i\geq K.$$ By passing to the limit, $i\to\infty,$ noting that \eqref{def:mcv} and \eqref{def:wcv},
 $$\nabla_{p_\infty x}\Delta^{G_\infty} f\geq K.$$ This proves the result.
\end{proof}

\section{Bounded harmonic functions}\label{sec:harmonic functions}
Let $G=(V,E,m,w)$ be a weighted graph. Fix a vertex $x_0\in V.$ For a sequence of vertices $\{x_i\}_{i=1}^\infty,$ we write $x_i\to \infty$ if $$d(x_i,x_0)\to\infty, i\to \infty.$$ 
For a function $f\in C(V)$ and $A\in \R,$ we say that $$f(x)\to A\quad \mathrm{as}\ x\to \infty,\quad \mathrm{denoted\ by} \lim_{x\to \infty} f(x)=A,$$ if for all $\epsilon>0,$ there exists $N>0$ such that 
$|f(x)-A|<\epsilon$ for all $x\in V\setminus B_N(x_0)$. Similarly, we may also define $\liminf_{x\to \infty}f$ and $\limsup_{x\to \infty}f.$

For any $x,y\in V,$ $x\neq y$, and any function $f\in C(V),$ we recall that $\nabla_{xy}f=\frac{f(x)-f(y)}{d(x,y)}.$ Moreover, by the triangle inequality, $$|\nabla f|_\infty =\sup_{x,y\in V, x\neq y} |\nabla_{xy}f|.$$ For any edge $e=\{x,y\}\in E,$ we write $|\nabla_e f|=|\nabla_{xy}f|.$ 
For a function $f\in C(V)$ and $A\in \R,$ we say that $$|\nabla_e f|\to A,\ e\to\infty$$ if for all $\epsilon>0,$ there exists $N>0$ such that for all $e=\{x,y\}\in E$ satisfying $\min\{d(x,x_0),d(y,x_0)\}\geq N,$ $$||\nabla_{e}f|-A|<\epsilon.$$

\subsection{Bakry-\'Emery curvature}

In this subsection, we aim to prove that the space of bounded harmonic functions is of finite dimension if the Bakry-\'Emery curvature is nonnegative outside a ball (see Theorem~\ref{thm:mainBE}).

\begin{lemma}\label{lem:gradto01} Let $(G,x_0)\in \BG^p$ satisfy
$$\mathcal{K}^{G}_{\infty}(x)\geq 0,\quad \forall x\in V\setminus B_{R_0}(x_0),$$ for some $x_0\in V$ and $R_0\in \N.$ Then for any bounded harmonic function $u$ on $G,$ $$\Gamma(u)(x)\to 0,\quad x\to \infty.$$
\end{lemma}
\begin{proof} Suppose it is not true, then there exist $\{x_i\}_{i=1}^\infty\subset V,$ $x_i\to \infty,$ and a positive constant $c$ such that $\Gamma(u)(x_i)\geq c$ for all $i\in \N.$
Consider the sequence of rooted graphs $\{(G_i,x_i)\}_{i=1}^\infty$ with $G_i=G$ and functions $u_i:G_i\to\R,$ $u_i=u,$ for all $i\geq 1.$ By Theorem~\ref{thm:comp2}, there exist a subsequence, still denoted by $\{(G_i,x_i)\}_{i=1}^\infty,$ $(G_\infty,x_\infty)\in \BG^p,$ and $u_\infty:V_\infty\to\R,$ such that
$$(G_i,x_i)\xrightarrow[]{pGH}(G_\infty,x_\infty), \quad\mathrm{and} \quad u_i\to u_\infty,\quad i\to \infty.$$

For any $R\in \N,$ and sufficiently large $i,$ there exist rooted graph isomorphisms
$$\varphi_{i,R}:(B_R^{G_i}(x_i),x_i)\to(B_R^{G_\infty}(x_\infty),x_\infty),$$ and
$$u_i\circ \varphi_{i,R}^{-1}\to u_\infty,\quad \mathrm{on}\ B_R^{G_\infty}(x_\infty).
$$ For $R\geq 2,$ by the convergence of weights, \eqref{def:mcv} and \eqref{def:wcv},  we have
$$\Gamma^{G_i}(u_i)(x_i)\to \Gamma^{G_\infty}(u_\infty)(x_\infty),\quad i\to \infty.$$ Hence \begin{equation}\label{eq:pt1}\Gamma^{G_\infty}(u_\infty)(x_\infty)\geq c.\end{equation} Since $u_i$ are harmonic functions on $G_i,$ $u_\infty$
is harmonic on $G_\infty.$ Moreover, $\|u_\infty\|_\infty\leq \|u_i\|_\infty=\|u\|_\infty<\infty.$

 We claim that $G_\infty$ has nonnegative Bakry-\'Emery curvature everywhere. For any $y_\infty\in V_\infty,$ set $R=2d(x_\infty,y_\infty)+2.$ By the pointed Gromov-Hausdorff convergence for the scale $R$, setting $y_i=\varphi_{i,R}^{-1}(y_\infty),$ we get $$(B_2^{G_i}(y_i),y_i)\to(B_2^{G_\infty}(y_\infty),y_\infty),\quad i\to\infty,$$ where the rooted graph isomorphism is given by $\varphi_{i,R}\big|_{B_2^{G_i}(y_i)}.$ Since $x_i\to \infty$ in $G,$ $B_2^{G_i}(y_i)\cap B_{R_0}(x_0)=\emptyset$ for sufficiently large $i,$ which yields that $\mathcal{K}_{\infty}^{G_i}(y_i)\geq 0.$ By Proposition~\ref{prop:lowerbound1}, we obtain that
 $$\mathcal{K}_{\infty}^{G_\infty}(y_\infty)\geq 0.$$
 
 Hence $u_\infty$ is a bounded harmonic function on $G_\infty,$ which has nonnegative Bakry-\'Emery curvature everywhere. By the Liouville theorem in \cite{hua2019Liouville}, 
 we obtain that $u_\infty$ is constant. This yields that $\Gamma^{G_\infty}(u_\infty)(x_\infty)=0$ which contradicts \eqref{eq:pt1}.
This proves the lemma. 
\end{proof}

\begin{proof}[Proof of Theorem~\ref{thm:mainBE}]
For any bounded harmonic function $u$ on $G,$ we claim that
\begin{equation}\label{eq:BE1}\sup_{x\in V}\Gamma(u)(x)=\max_{x\in B_{R_0}}\Gamma(u)(x).\end{equation} By the curvature condition and the harmonicity of $u,$ 
$$\Gamma_2(u)(x)=\frac{1}{2}\Delta \Gamma(u)(x)\geq 0,\quad \forall x\in V\setminus B_{R_0}(x_0).$$ Hence $\Gamma(u)$ is a subharmonic function on $V\setminus B_{R_0}(x_0).$ By applying the maximum principle on $V\setminus B_{R_0}(x_0),$ noting that $\Gamma(u)(x)\to 0,\ x\to\infty,$ 
$$\sup_{x\in V\setminus B_{R_0}(x_0)}\Gamma(u)(x)=\max_{x\in S_{R_0}(x_0)}\Gamma(u)(x).$$ This proves the claim.

We denote by $$\mathcal{H}(B_{R_0}(x_0)):=\{f:B_{R_0+1}(x_0)\to \R: \Delta f(x)=0,\forall x\in B_{R_0}(x_0)\}$$ the space of harmonic functions on $B_{R_0}(x_0).$ By the maximum principle, $\mathcal{H}(B_{R_0}(x_0))$ is linearly isomorphic to $\R^{S_{R_0+1}(x_0)},$ the space of functions on $S_{R_0+1}(x_0).$ Hence $$\dim \mathcal{H}(B_{R_0}(x_0))=\sharp{S_{R_0+1}}(x_0).$$

We define a linear operator $$T:\mathcal{H}_0(G)\to \mathcal{H}(B_{R_0}(x_0)),\ f\mapsto f\big|_{B_{R_0+1}(x_0)}.$$ By \eqref{eq:BE1}, $T$ is injective. Hence $$\dim \mathcal{H}_0(G)\leq  \dim\mathcal{H}(B_{R_0}(x_0)).$$ This proves the theorem.
\end{proof}

\subsection{Ollivier curvature}

We now prove finite-dimensionality of bounded harmonic functions in case of nonnegative Ollivier curvature outside a ball (see Theorem~\ref{thm:mainOllivier}).
We first show that the gradient of a bounded harmonic function tends to zero by the Liouville property, similar to the Bakry-\'Emery curvature case.
After that, we will prove a maximum principle for the gradient of a harmonic function.

\begin{lemma}\label{lem:gradto02} Let $(G,x_0)\in \BG^p$ satisfy
$$\kappa(e)\geq 0,\quad \forall e\in E_{V\setminus B_{R_0}(x_0)},$$ for some $x_0\in V$ and $R_0\in \N.$ Then for any bounded harmonic function $u$ on $G,$ $$|\nabla_e u|\to 0,\quad e\to \infty.$$
\end{lemma}
\begin{proof} Suppose that it is not true, then there exists $\{e_i\}_{i=1}^\infty\subset E,$ $e_i=\{x_i,y_i\},$ satisfying $d(x_i,x_0)\to \infty, i\to \infty,$ such that
$$|\nabla_{e_i} u|\geq c> 0.$$ Consider the sequence $\{(G_i,x_i)\}_{i=1}^\infty$ with $G_i=G$ and functions $u_i:V_i\to\R,$ $u_i=u,$ for all $i\geq 1.$ By Theorem~\ref{thm:comp2}, there exist a subsequence, still denoted by $\{(G_i,x_i)\}_{i=1}^\infty,$ $(G_\infty,x_\infty)\in \BG^p,$ and $u_\infty:V_\infty\to\R,$ such that
$$(G_i,x_i)\xrightarrow[]{pGH}(G_\infty,x_\infty), \quad\mathrm{and}\quad u_i\to u_\infty,\quad i\to \infty.$$ By passing to the subsequence, there exists $y_\infty\in V_\infty,$ $y_\infty\sim x_\infty,$ such that $\varphi_{i,R}(y_i)=y_\infty$ ($R\geq 2$) for sufficiently large $i.$ Hence $$|\nabla_{e_i}u_i|\to |\nabla_{x_\infty y_\infty}u_\infty|, i\to \infty,$$ which implies that $|\nabla_{x_\infty y_\infty} u_\infty|\geq c.$  As in the proof of Lemma~\ref{lem:gradto01}, one can show that $G_\infty$ is an infinite graph with nonnegative Ollivier curvature by Proposition~\ref{prop:lowerbound2} and $d(x_i,x_0)\to \infty, i\to \infty.$ Moreover,  $u_\infty$ is a bounded harmonic function on $G_\infty.$ Hence the Liouville theorem in \cite{jost2019liouville} implies that
$u_\infty$ is constant.  This contradicts $|\nabla_{x_\infty y_\infty} u_\infty|\geq c$
and proves the lemma.
\end{proof}

\begin{lemma}[Maximum principle for the gradient]\label{lem:MaxPrincipleOllivierGradient} Let $G=(V,E,m,w)$ be a weighted graph and $W \subset V$ be finite. Suppose $\kappa(x,y)\geq 0$ for all $x,y\in W\setminus \delta W$ with $y\sim x$. Let $u:V \to \R$ be harmonic on $W\setminus \delta W$. Then, 
\[
\max_{x\in W, y \sim x} |\nabla_{xy} u| = \max_{x\in \delta W,  y\sim x}|\nabla_{xy}u|
\]
where $\delta W := \{x \in W: x \sim y $ for some $ y \in V\setminus W\}$.
\end{lemma}
\begin{proof}
Let 
\[
M := \max_{x\in W, y \sim x} |\nabla_{xy} u|.
\]
We argue by contradiction. Suppose that the maximum is not attained at any boundary point $x \in \delta W$. Then the maximum is attained for some $x,y \in W \setminus \delta W$.
Let $\widetilde G$ be the graph consisting of all edges containing a vertex from $W$ with the inherited weights from $G.$ Then, $u$ is $M$-Lipschitz w.r.t the combinatorial graph distance of $\widetilde G$. Moreover, $\kappa(x,y) \geq 0$ within $\widetilde G$ for all $x\sim y \in W \setminus \delta W$.

Let $x_0,\ldots,x_n \in W$ be a geodesic w.r.t. $\widetilde G$ of maximal length s.t. $\nabla_{x_ix_{i+1}}u = M$. By assumption, we have $x_i \in W \setminus \delta W$. Since $\{x_i\}_i$ is a geodesic within $W$, we have $\kappa(x_0,x_n)\geq 0$. By \cite[Lemma~2.3]{jost2019liouville} applied to $\widetilde G$, there exist $x\sim x_0$ and $y \sim x_n$ s.t. $\nabla_{xy}u = M$ and $d_{\widetilde G}(x,y) > n$. 
As $x_0,x_n \in W \setminus \delta W$, we have $x,y \in W$.
Let $\gamma$ be a geodesic w.r.t. $\widetilde G$ from $x$ to $y$. If $\gamma \subset W$, then this is a contradiction to maximal length of $\{x_i\}_i$ as $d_{\widetilde G}(x,y)>n$. Thus, there exist $v\sim w$ as consecutive vertices of $\gamma$ with $v\notin W$ and $w \in \delta W$. Moreover, $|\nabla_{vw}u|=M,$ which contradicts that the maximal gradient is not attained at the boundary. This proves the lemma.
\end{proof}

With the maximum principle in hands, we can now show that the space of bounded harmonic functions has finite dimension.

\begin{proof}[Proof of Theorem~\ref{thm:mainOllivier}]
Let $u$ be a bounded harmonic function.
By Lemma~\ref{lem:MaxPrincipleOllivierGradient}
and Lemma~\ref{lem:gradto02}, {with $W=B_R(x_0)\setminus B_{R_0-1}(x_0)$ as $R\to \infty,$} we have
\[
\sup_{\substack{x\in V\setminus B_{R_0-1}(x_0)  \\ y\sim x}} |\nabla_{xy}u|= \max_{\substack{x\in S_{R_0}(x_0)\\y\sim x}} |\nabla_{xy}u|
\]
In particular, if $u$ vanishes on $B_{R_0+1}$, then it vanishes everywhere. Moreover by harmonicity, the function $u$ on $B_{R_0+1}(x_0)$ is uniquely determined by its values on $S_{R_0+1}(x_0)$. Similarly to the proof of Theorem~\ref{thm:mainBE}, this shows that
\[
\dim \mathcal{H}_0(G)\leq \sharp{S_{R_0+1}}(x_0).
\]
This finishes the proof.
\end{proof}

\section{Ends and bounded harmonic functions}
In this section, we develop the theory of harmonic functions on ends of graphs, see e.g. \cite{Li12} for the Riemannian setting. Let $G=(V,E,m,w)$ be a weighted graph. For any subset {$K\subset V,$} we denote by $\partial K:=\{y\in V\setminus K: d(x,K)=1\},$ where $d(x,K)=\inf_{z\in K}d(x,z),$ the vertex boundary of $K,$ and by $\overline{K}=K\cup \partial K.$ For any $\Omega\subset V,$ we write $1_\Omega$ as the indicator function on $\Omega.$
 A weighted graph $G$ is called \emph{non-parabolic} (or transient) if it admits a positive Green's function on $V.$ Otherwise, it is called \emph{parabolic} (or recurrent). It is well-known that $G$ is non-parabolic if and only if it admits a non-constant, positive, superharmonic function on $V.$ For any $\rho\in \N$ and $x_0,x_1\in V$ with $d(x_1,x_0)\leq\rho,$ let $\Gamma_\rho(\cdot,x_1)$ be the Green's function on $B_\rho(x_0)$ with Dirichlet boundary condition, i.e. 
\begin{equation}\label{eq:hh1}\left\{\begin{array}{l} \Delta \Gamma_{\rho}(\cdot,x_1)=-\frac{1}{m(x_1)}1_{x_1}(\cdot),\\
 \Gamma_{\rho}(\cdot,x_1)\big|_{\partial B_\rho(x_0)}=0.\end{array}\right.\end{equation}
By the maximum principle, $\Gamma_{\rho}(\cdot,x_1)$ is non-decreasing in $\rho,$ so that we define for any $x\in V,$ \begin{equation}\label{eq:co1}\Gamma(x,x_1)=\lim_{\rho\to\infty}\Gamma_{\rho}(x,x_1).\end{equation} As is well-known, $G$ is non-parabolic if and only if $\Gamma(x,x_1)<\infty$ for some (hence all) $x, x_1.$ 

For any finite $\Omega\subset V,$ we consider $V\setminus \Omega$ as the induced subgraph on $V\setminus \Omega$ with vertex and edge weights induced from $G.$ Any infinite connected component $\Pi$ of $V\setminus \Omega$ is called an end of $G$ w.r.t. $\Omega.$ We denote by $N_{\Omega}(G)$ the number of ends of $G$ w.r.t. $\Omega,$ which is finite since the graph $G$ is locally finite. In general, when we say $\Pi$ is an end we mean that $\Pi$ is an end of $G$ w.r.t. some finite $\Omega.$ One easily sees that for $\Omega_1\subset \Omega_2,$ $$N_{\Omega_1}(G)\leq N_{\Omega_2}(G).$$ The sequence of finite subsets $\{\Omega_i\}_{i=1}^\infty$ is called an exhaustion of $G$ if $\Omega_i\subset \Omega_{i+1},$ $i\in \N,$ and $V=\cup_{i=1}^\infty \Omega_i.$ We define
$$N(G)=\lim_{i\to\infty}N_{\Omega_i}(G)$$ as the number of ends of $G.$ Note that the above limit does not depend on the choice of the exhaustion, so that it is well-defined. If $N(G)<\infty,$ then there exists some finite $\Omega\subset V$ such that $N(G)=N_{\Omega}(G).$

For an end $\Pi,$ $\overline{\Pi}$ is regarded as the induced graph on $\overline{\Pi}$ with the inherited weights. 
\begin{defi} An end $\Pi$ is called non-parabolic if there exists $f:\overline{\Pi}\to (0,+\infty)$ satisfying $\Delta f=0$ on $\Pi,$ $f\big|_{\partial \Pi}\equiv 1$ and $$\liminf_{x\to \Pi(\infty)}f<1,$$ where $\liminf_{x\to \Pi(\infty)}$ is understood as $x\to \infty$ and $x\in \Pi.$ Here $f$ is called a barrier function on $\Pi.$ Otherwise, $\Pi$ is called parabolic.
\end{defi} 

One can show that $\Pi$ is non-parabolic if and only if $\overline{\Pi},$ as an induced subgraph, is non-parabolic. For a non-parabolic end $\Pi,$ by extending $f\equiv 1$ to $V\setminus \overline{\Pi},$ we get a non-constant superharmonic function on $V.$  For any $\rho\in \N,$ we denote $\Pi_\rho:=B_\rho\cap \Pi.$ Let $f_\rho$ satisfy 
\[\left\{\begin{array}{ll} \Delta f_{\rho}(x)=0,& x\in \Pi_\rho\\
f_\rho\big|_{\partial \Pi}=1,&\\
f_\rho\big|_{\partial B_\rho\cap \Pi}=0.&\\
\end{array}\right.\]  By the monotonicity of $f_\rho,$ we define $$f=\lim_{\rho\to\infty} f_\rho.$$ One easily checks that $f$ is a barrier function on $\Pi.$ In fact, $f$ is the minimal barrier function on $\Pi.$ By the maximum principle, for any $x_1\in V,$ 
there exists a constant $C$ such that \begin{equation}\label{eq:bcom}\Gamma(x,x_1)\leq C f(x),\quad \forall x\in \Pi.\end{equation} One can show that $G$ is non-parabolic if and only if for some (hence for all) finite $\Omega\subset V$ there exists a non-parabolic end w.r.t. $\Omega.$ 

Let $\Pi$ be an end w.r.t. $\Omega$ and $\Omega'$ be a finite subset containing $\Omega.$ Then one can prove the following: \begin{enumerate}
\item If $\Pi$ is non-parabolic, then there exists a non-parabolic end $\Pi'$ w.r.t. $\Omega'$ contained in $\Pi\setminus \Omega'.$
\item If $\Pi$ is parabolic, then all ends w.r.t. $\Omega'$ contained in $\Pi\setminus \Omega'$ are parabolic.
\end{enumerate} For any finite $\Omega\subset V,$ we denote by $N_{\Omega}^0(G)$ (resp.$N_{\Omega}'(G)$) the number of non-parabolic (resp. parabolic) ends of $G$ w.r.t. $\Omega.$ By the above property, we define
$$N^0(G)=\lim_{i\to\infty}N_{\Omega_i}^0(G),\ N'(G)=\lim_{i\to\infty}N_{\Omega_i}'(G)$$ as the number of non-parabolic and parabolic ends of $G$ respectively, where $\{\Omega_i\}$ is an exhaustion of $G.$ Here, they are independent of the choice of the exhaustion.

The following theorem is a discrete analog of \cite[Theorem~21.1]{Li12}.

\begin{thm}\label{thm:harm1} Let $G$ be a non-parabolic graph, $\Omega$ be a finite subset of $V$ and $x_0\in\Omega.$ Let {$g:V\to \R$} be  harmonic on $V\setminus\Omega.$ Then there exist a harmonic function $h$ on $V$ and a constant $C$ such that
$$|g(x)-h(x)|\leq C \Gamma(x,x_0),\quad \forall x\in V\setminus\Omega.$$ Moreover, if $g$ is bounded, then $\|h\|_{\infty}\leq \|g\|_{\ell^\infty(V\setminus\Omega)}.$
\end{thm}

\begin{proof}
We define
{\[
h:=g+\sum_{x \in V} \Delta g(x) \Gamma(x,\cdot) m(x).
\] Noting that $\Delta g$ has finite support in $\Omega,$ the above is a finite sum over $\Omega.$}
As $\Delta \Gamma(x,\cdot)m(x) = -1_x$, we obtain $\Delta h = 0$.
As $\frac{\Gamma(y,\cdot)}{\Gamma(y,y)}$ is the minimal positive superharmonic function attaining value one at $y$, and as $\frac{\Gamma(x,\cdot)}{\Gamma(x,y)}$ is also positive superharmonic and attaining value one at $y$, we get
\[
\frac{\Gamma(y,\cdot)}{\Gamma(y,y)} \leq \frac{\Gamma(x,\cdot)}{\Gamma(x,y)}
\] 
implying $\frac{\Gamma(x,\cdot)}{\Gamma(y,\cdot)} \in \ell_\infty(V)$.
Particularly,
\[
|h-g| \leq 
 \sum_{x \in \Omega} |\Delta g(x)| \Gamma(x,\cdot) m(x) \leq C\Gamma(x_0,\cdot)
\]
for some $C>0$. This proves the first assertion. 

{We now prove the second assertion. For any $\rho\geq 1,$ let $h_\rho$ solve the following problem
\[\left\{\begin{array}{ll} \Delta h_{\rho}(x)=0,& x\in B_\rho,\\
h_\rho\big|_{\partial B_\rho}=g.&\\
\end{array}\right.\] This yields that $\|h_\rho\|_\infty\leq\|g\|_{\ell^\infty(\partial B_\rho)}\leq \|g\|_{\ell^\infty(V\setminus\Omega)}.$ Note that $$h_\rho=g+\sum_{x \in V} \Delta g(x) \Gamma_\rho(x,\cdot) m(x).$$
Hence by \eqref{eq:co1}, $\lim_{\rho\to\infty} h_\rho=h.$} This proves the second assertion and finishes the proof.


\end{proof}

\begin{thm}\label{thm:bn1} Let $G$ be a non-parabolic weighted graph.There exists a subspace $\mathcal{K}_0(G)$ of $\mathcal{H}_0(G)$ such that
$$\dim \mathcal{K}_0(G)=N^0(G).$$
\end{thm}
\begin{proof} Since $G$ is non-parabolic, $N^0(G)\geq 1.$ For the case that $N^0(G)=1,$ it suffices to choose $\mathcal{K}_0(G)$ consisting of constant functions. Now we consider the case that $N^0(G)\geq 2.$ For any $\rho\in \N, $ let $\Pi_1,\Pi_2,\cdots, \Pi_{N},$ $N=N_{B_{\rho}(x_0)}^0(G),$ be non-parabolic ends w.r.t. $B_{\rho}(x_0).$ We want to construct an $N$-dimensional subspace of $\mathcal{H}_0(G).$ Since we can do that for arbitrary $\rho,$ the result follows.

For any $1\leq i\leq N,$ we define $g_i=1_{\overline{\Pi_i}}.$ By Theorem~\ref{thm:harm1},
there exist a harmonic function $h_i$ on $V$ and a constant $C_i$ such that
$$|g_i(x)-h_i(x)|\leq C_i \Gamma(x,x_0),\quad \forall x\in V\setminus B_{\rho}(x_0).$$
For any $\Pi_i,$ we denote by $f_i$ the minimal barrier function on $\Pi_i.$ Note that 
$$\liminf_{x\to \Pi_i(\infty)}f_i=0.$$ Hence, for any $1\leq i\leq N,$ there exists a sequence $\{x^i_l\}_{l=1}^{\infty}\subset \Pi_i$ such that $d(x^i_l, x_0)\to \infty, l\to \infty$ such that
$$\lim_{l\to \infty}f_i(x^i_l)=0.$$ By \eqref{eq:bcom}, there exists a large constant $C$ such that for any $1\leq i,j\leq N,$
$$|g_i(x)-h_i(x)|\leq C f_j(x),\quad \forall x\in \Pi_j.$$
Hence we obtain that
$$\lim_{l\to \infty}h_i(x^j_l)=\delta_{ij},$$ where $\delta_{ij}=1$ for $i=j,$ and $\delta_{ij}=0$ otherwise. Therefore $\{h_i\}_{i=1}^N$ is linearly independent. This proves the theorem.
\end{proof}

\begin{coro}\label{coro:ct1} For any weighted graph $G,$ $$N^0(G)\leq \dim \mathcal{H}_0(G).$$
\end{coro}
\begin{proof} If $G$ is parabolic, then the statement is trivial. If $G$ is non-parabolic, then by Theorem~\ref{thm:bn1},
$$N^0(G)\leq \dim \mathcal{K}_0(G)\leq\dim \mathcal{H}_0(G).$$ This proves the corollary.
\end{proof}

Now we are ready to prove Corollary~\ref{coro:finiteend}.
\begin{proof}[Proof of Corollary~\ref{coro:finiteend}] The corollary follows from Theorem~\ref{thm:mainBE} and Corollary~\ref{coro:ct1}.
\end{proof}

\bigskip
\bigskip

\textbf{Acknowledgements.} We thank the MPI MiS Leipzig and Fudan University for their hospitality.  B. Hua is supported by NSFC, no.11831004, and by Shanghai Science and Technology Program [Project No. 22JC1400100].

\bibliographystyle{alpha}
\bibliography{nonnegative_outside_no_spectra}

\begin{thebibliography}{LMPR19}

\bibitem[AL07]{AldousLyons07}
D.~Aldous and R.~Lyons.
\newblock Processes on unimodular random networks.
\newblock {\em Electron. J. Probab.}, 12:no. 54, 1454--1508, 2007.

\bibitem[And83]{Anderson83}
M.~Anderson.
\newblock {The Dirichlet problem at infinity for manifolds of negative
  curvature}.
\newblock {\em J. Differential Geom.}, 18(4):701--721, 1983.

\bibitem[Bak87]{Bakry87}
D.~Bakry.
\newblock {\'Etude des transformations de Riesz dans les vari\'et\'es
  Riemanniennes \`a courbure de Ricci minor\'ee}.
\newblock In {\em {S\'eminaire de Probabilit\'es, XXI}}, volume 1247 of {\em
  Lecture Notes in Math.}, pages 137--172, Berlin, 1987. Springer.

\bibitem[BBI01]{Burago01}
D.~Burago, Yu. Burago, and S.~Ivanov.
\newblock {\em A course in metric geometry}.
\newblock Number~33 in Graduate Studies in Mathematics. American Mathematical
  Society, Providence, RI, 2001.

\bibitem[BE85]{BakryEmery85}
D.~Bakry and M.~\'Emery.
\newblock Diffusions hypercontractives.
\newblock In {\em S\'eminaire de Probabilit\'es, XIX (1983/84)}, volume 1123 of
  {\em Lecture Notes in Math.}, pages 177--206, Berlin, 1985. Springer.

\bibitem[Ben91]{Benjamini91}
I.~Benjamini.
\newblock {Instability of the Liouville property for quasi-isometric graphs and
  manifolds of polynomial volume growth}.
\newblock {\em J. Theoret. Probab.}, 4(3):631--637, 1991.

\bibitem[BGL14]{BakryGentilLedoux}
D.~Bakry, I.~Gentil, and M.~Ledoux.
\newblock {\em {Analysis and geometry of Markov diffusion operators}}.
\newblock Number 348 in Grundlehren der Mathematischen Wissenschaften.
  Springer, Cham, 2014.

\bibitem[Bri13]{Brighton13}
K.~Brighton.
\newblock {A Liouville-type theorem for smooth metric measure spaces}.
\newblock {\em J. Geom. Anal.}, 23(2):562--570, 2013.

\bibitem[BS01]{BenSchramm01}
I.~Benjamini and O.~Schramm.
\newblock Recurrence of distributional limits of finite planar graphs.
\newblock {\em Electron. J. Probab.}, 6:no. 23, 13, 2001.

\bibitem[Che]{Chengpreprint}
S.Y. Cheng.
\newblock Finite dimensionality of the spaces of positive and bounded harmonic
  functions.
\newblock {\em preprint}.

\bibitem[CY75]{ChengYau75}
S.~Y. Cheng and S.~T. Yau.
\newblock {Differential equations on Riemannian manifolds and their geometric
  applications}.
\newblock {\em Comm. Pure Appl. Math.}, 28(3):333--354, 1975.

\bibitem[Don86]{Donnelly86}
H.~Donnelly.
\newblock Bounded harmonic functions and positive {R}icci curvature.
\newblock {\em Math. Z.}, 191(4):559--565, 1986.

\bibitem[Elw91]{Elw91}
K.~D. Elworthy.
\newblock Manifolds and graphs with mostly positive curvatures.
\newblock In {\em Stochastic analysis and applications ({L}isbon, 1989)},
  volume~26 of {\em Progr. Probab.}, pages 96--110. Birkh\"{a}user Boston,
  Boston, MA, 1991.

\bibitem[Ers04]{Erschler04}
A.~Erschler.
\newblock {Liouville property for groups and manifolds}.
\newblock {\em Invent. Math.}, 155(1):55--80, 2004.

\bibitem[For03]{forman2003bochner}
R.~Forman.
\newblock {Bochner's method for cell complexes and combinatorial Ricci
  curvature}.
\newblock {\em Discrete and Computational Geometry}, 29(3):323--374, 2003.

\bibitem[Gri90]{Grigoryan90}
A.~Grigor'yan.
\newblock {Dimension of spaces of harmonic functions}.
\newblock {\em Math. Notes}, 48:1114--1118, 1990.

\bibitem[Gri91]{Grigoryan91}
A.~Grigor'yan.
\newblock {The heat equation on noncompact Riemannian manifolds (Russian)}.
\newblock {\em Mat. Sb.}, 182(1):55--87 (English translation in Math. USSR--Sb.
  72(1): 47--77, 1992.), 1991.

\bibitem[Gro81]{Gromov81}
M.~Gromov.
\newblock Groups of polynomial growth and expanding maps.
\newblock {\em Inst. Hautes \'{E}tudes Sci. Publ. Math.}, (53):53--73, 1981.

\bibitem[HM21]{HuaMun21}
B.~Hua and F.~M\"{u}nch.
\newblock Every salami has two ends.
\newblock {\em arXiv:2105.11887}, 2021.

\bibitem[Hua19]{hua2019Liouville}
B.~Hua.
\newblock Liouville theorem for bounded harmonic functions on manifolds and
  graphs satisfying non-negative curvature dimension condition.
\newblock {\em Calc. Var. Partial Differential Equations}, 58(2):Art. 42, 8,
  2019.

\bibitem[JM21]{JostMun21}
J.~Jost and F.~M\"{u}nch.
\newblock Characterizations of forman curvature.
\newblock {\em arXiv:2110.04554}, 2021.

\bibitem[JMR19]{jost2019liouville}
J.~Jost, F.~M{\"u}nch, and C.~Rose.
\newblock {Liouville property and non-negative Ollivier curvature on graphs}.
\newblock {\em arXiv preprint arXiv:1903.10796}, 2019.

\bibitem[Kai96]{Kaimanovich96}
V.~A. Kaimanovich.
\newblock {Boundaries of invariant Markov operators: the identification
  problem}.
\newblock In M.~Pollicott and K.~Schmidt, editors, {\em Ergodic theory of
  $\Z^d$ actions (Warwick, 1993--1994)}, London Math. Soc. Lecture Note Ser.
  228., pages 127--176, Cambridge, 1996. Cambridge Univ. Press.

\bibitem[KV83]{KaimanovichVershik83}
V.~A. Kaimanovich and A.~M. Vershik.
\newblock {Random walks on discrete groups: boundary and Entropy}.
\newblock {\em Ann. Probab.}, 11(3):457--490, 1983.

\bibitem[Li12]{Li12}
P.~Li.
\newblock {\em {Geometric analysis}}, volume 134 of {\em Cambridge Studies in
  Advanced Mathematics}.
\newblock Cambridge University Press, Cambridge, 2012.

\bibitem[LLY11]{LinLuYau11}
Y.~Lin, L.~Lu, and S.-T. Yau.
\newblock Ricci curvature of graphs.
\newblock {\em Tohoku Math. J. (2)}, 63(4):605--627, 2011.

\bibitem[LMPR19]{liu2019distance}
S.~Liu, F.~M{\"u}nch, N.~Peyerimhoff, and C.~Rose.
\newblock Distance bounds for graphs with some negative bakry-emery curvature.
\newblock {\em Analysis and Geometry in Metric Spaces}, 7(1):1--14, 2019.

\bibitem[LT87]{LiTam87}
P.~Li and L.-F. Tam.
\newblock {Positive harmonic functions on complete manifolds with nonnegative
  curvature outside a compact set}.
\newblock {\em Ann. of Math. (2)}, 125(1):171--207, 1987.

\bibitem[LT92]{LiTam92}
P.~Li and L.-F. Tam.
\newblock Harmonic functions and the structure of complete manifolds.
\newblock {\em J. Differential Geom.}, 35(2):359--383, 1992.

\bibitem[LY86]{LiYau86}
P.~Li and S.~T. Yau.
\newblock {On the parabolic kernel of the Schroedinger operator}.
\newblock {\em Acta Math.}, 156(3--4):153--201, 1986.

\bibitem[LY10]{LinYau10}
Y.~Lin and S.~T. Yau.
\newblock {Ricci curvature and eigenvalue estimate on locally finite graphs}.
\newblock {\em Math. Res. Lett.}, 17(2):343--356, 2010.

\bibitem[Lyo87]{Lyons87}
T.~Lyons.
\newblock {Instability of the Liouville property for quasi-isometric Riemannian
  manifolds and reversible Markov chains}.
\newblock {\em J. Differential Geom.}, 26(1):33--66, 1987.

\bibitem[MR20]{munch2020spectrally}
F.~M{\"u}nch and C.~Rose.
\newblock {Spectrally positive Bakry-Emery Ricci curvature on graphs}.
\newblock {\em Journal de Math{\'e}matiques Pures et Appliqu{\'e}es},
  143:334--344, 2020.

\bibitem[M{\"u}n18]{munch2018perpetual}
F.~M{\"u}nch.
\newblock {Perpetual cutoff method and discrete Ricci curvature bounds with
  exceptions}.
\newblock {\em arXiv:1812.02593}, 2018.

\bibitem[M{\"u}n19]{munch2019non}
F.~M{\"u}nch.
\newblock {Non-negative Ollivier curvature on graphs, reverse Poincar\'e
  inequality, Buser inequality, Liouville property, Harnack inequality and
  eigenvalue estimates}.
\newblock {\em arXiv preprint arXiv:1907.13514}, 2019.

\bibitem[M{\" u}n22]{Mun22F}
F.~M{\" u}nch.
\newblock Reflective graphs, ollivier curvature, effective diameter, and
  rigidity.
\newblock {\em arXiv:2205.15857}, 2022.

\bibitem[MW19]{munch2019Ollivier}
F.~M\"{u}nch and R.~K. Wojciechowski.
\newblock Ollivier {R}icci curvature for general graph {L}aplacians: {H}eat
  equation, {L}aplacian comparison, non-explosion and diameter bounds.
\newblock {\em Adv. Math.}, 356:106759, 45, 2019.

\bibitem[Oll09]{Ollivier09}
Y.~Ollivier.
\newblock Ricci curvature of {M}arkov chains on metric spaces.
\newblock {\em J. Funct. Anal.}, 256(3):810--864, 2009.

\bibitem[Pae12]{Paeng12}
S.-H. Paeng.
\newblock Volume and diameter of a graph and {O}llivier's {R}icci curvature.
\newblock {\em European J. Combin.}, 33(8):1808--1819, 2012.

\bibitem[Sal22]{salez2022sparse}
Justin Salez.
\newblock Sparse expanders have negative curvature.
\newblock {\em Geometric and Functional Analysis}, pages 1--28, 2022.

\bibitem[SC92]{Saloff92}
L.~Saloff-Coste.
\newblock {A note on Poincar\'e, Sobolev, and Harnack inequalities}.
\newblock {\em Internat. Math. Res. Notices}, (2):27--38, 1992.

\bibitem[Sch99]{Schmuck96}
M.~Schmuckenschl\"ager.
\newblock Curvature of nonlocal {M}arkov generators.
\newblock In {\em Convex geometric analysis ({B}erkeley, {CA}, 1996)},
  volume~34 of {\em Math. Sci. Res. Inst. Publ.}, pages 189--197. Cambridge
  Univ. Press, Cambridge, 1999.

\bibitem[STW00]{SungTamWang00}
C.-J. Sung, L.-F. Tam, and J.~Wang.
\newblock Spaces of harmonic functions.
\newblock {\em J. London Math. Soc. (2)}, 61(3):789--806, 2000.

\bibitem[Sul83]{Sullivan83}
D.~Sullivan.
\newblock {The Dirichlet problem at infinity for a negatively curved manifold}.
\newblock {\em J. Differential Geom.}, 18(4):723--732, 1983.

\bibitem[TT21]{tee2021enhanced}
P.~Tee and C.~Trugenberger.
\newblock {Enhanced Forman curvature and its relation to Ollivier curvature}.
\newblock {\em Europhysics Letters}, 133(6):60006, 2021.

\bibitem[Wan02]{WangFY02}
F.~Y. Wang.
\newblock {Liouville theorem and coupling on negatively curved manifolds}.
\newblock {\em Stochastic Process. Appl.}, 100:27--39, 2002.

\bibitem[Woe09]{Woess09}
W.~Woess.
\newblock {\em Denumerable Markov chains--Generating functions, boundary
  theory, random walks on trees.}
\newblock EMS Textbooks in Mathematics. European Mathematical Society (EMS),
  Z\"urich, 2009.

\bibitem[Yau75]{Yau75}
S.~T. Yau.
\newblock {Harmonic functions on complete Riemannian manifolds}.
\newblock {\em Comm. Pure Appl. Math.}, 28:201--228, 1975.

\end{thebibliography}

\end{document}